\documentclass[pdftex]{amsart}
\usepackage{amsmath,amsthm,latexsym,amscd,amsbsy,amssymb,url}
\usepackage[all]{xypic}

\theoremstyle{plain}
\newtheorem{thm}{Theorem}[section]

\newtheorem{lem}[thm]{Lemma}
\newtheorem{pro}[thm]{Proposition}

\theoremstyle{definition}

\newtheorem{defin}[thm]{Definition}
\newtheorem{ex}[thm]{Example}

\newcommand{\nh}[1]{{\protect(NH$_{#1}$)}}
\newcommand{\ans}[1]{{\protect(A$n$S$_{#1}$)}}

\newcommand{\fne}[1]{{\protect(F$n$HE$_{#1}$)}}
\newcommand{\uv}[1]{{\protect(UV$^{n-1}_{#1}$)}}

\newcommand{\F}[1]{{\protect\mathcal{F}_{#1}}}

\newcommand{\N}[1]{{\protect\mathcal{N}_{#1}}}
\newcommand{\U}[1]{{\protect\mathcal{U}_{#1}}}
\newcommand{\V}[1]{{\protect\mathcal{V}_{#1}}}
\newcommand{\W}[1]{{\protect\mathcal{W}_{#1}}}
\renewcommand{\P}[1]{{\protect\mathcal{P}_{#1}}}
\newcommand{\Q}[1]{{\protect\mathcal{Q}_{#1}}}

\newcommand{\R}[1]{{\protect\mathcal{R}_{#1}}}

\DeclareMathOperator{\st}{st}
\DeclareMathOperator{\im}{im}

\begin{document}
  \bibliographystyle{abbrv}


\title[Near-homeomorphisms of N\"{o}beling manifolds]
{Near-homeomorphisms\\ of N\"{o}beling manifolds}

\author{A. Chigogidze}
\address{Department of Mathematics and Statistics,
University of North Carolina at Greensboro,
383 Bryan Bldg, Greensboro, NC, 27402, USA}
\email{chigogidze@uncg.edu}
\author{A. Nag\'{o}rko}
\address{Department of Mathematics and Statistics,
Univeristy of North Carolina at Greensboro,
387 Bryan Bldg, Greensboro, NC, 27402, USA}
\email{a\_nagork@uncg.edu}
\keywords{$n$-dimensional N\"obeling manifold, $Z$-set unknotting, near-homeomorphism}
\subjclass{Primary: 55M10; Secondary: 54F45}


\begin{abstract}
We characterize maps between $n$-dimensional N\"obeling manifolds that can be approximated by homeomorphisms.
\end{abstract}

\maketitle

\section{Introduction}
A long standing problem (see, for example, \cite[TC 10]{west}, \cite[Conjecture 5.0.5]{chibook}) of characterizing topologically universal $n$-dimensional N\"obeling space, as well as manifolds modeled on it, was solved recently by M.~Levin \cite{levin} and A.~Nag\'{o}rko \cite{nagorkophd}. Theory of N\"obeling manifolds, developed in \cite{levin}, \cite{levin2}, \cite{nagorkophd} based on completely different approaches, among other things contains various versions of $Z$-set unknotting theorem, open embedding theorem, $n$-homotopy classification theorem, etc. 

In this note we complete the picture by proving that for $n$-dimensional N\"obeling manifolds classes of
near-homeomorphisms, approximately $n$-soft maps, fine $n$-homotopy
equivalences and UV$^{n-1}$-mappings coincide. Recall that an
$n$-dimensional N\"obeling manifold is a Polish space locally
homeomorphic to $\nu^n$, the subset of $\mathbb{R}^{2n+1}$ consisting
of all points with at most $n$ rational coordinates.
 
\begin{defin}\label{D:relative}
  For each map~$f$ from a space~$X$ into a space~$Y$, for each open
  cover $\U{}$ of $Y$ and for each integer $n$, we define the
  following conditions.  \em\begin{enumerate}
  \item[\scriptsize\nh{\U{}}] There exists a homeomorphism of $X$ and
    $Y$ that is $\U{}$-close to~$f$.
  \item[\scriptsize\ans{\U{}}] For each at most $n$-dimensional metric
    space $B$, its closed subset $A$ and maps $\varphi$ and $\psi$
    such that the following diagram commutes
    \[
    \xymatrix{
      X \ar^{f}[r] & Y \\
      A \ar@^{(->}^{i}[r] \ar^{\varphi}[u] & B \ar_{\psi}[u]
      \ar@^{.>}_{
        k}[ul]\\
    }
    \]
    there exists a map $k \colon B \to X$ such that $k|A = \varphi$
    and $f \circ k$ is $\U{}$-close to $\psi$.
  \item[\scriptsize\fne{\U{}}] There exists a map $g$ from $Y$ to $X$
    such that $f \circ g$ is $\U{}$-$n$-homotopic\footnote{see
      section~\ref{S:preliminaries} for definitions.} to the identity
    on $Y$ and $g \circ f$ is $f^{-1}(\U{})$-$n$-homotopic to the
    identity on $X$ (with $f^{-1}(\U{})$ denoting $\{ f^{-1}(U) \}_{U
      \in \U{}}$).
  \item[\scriptsize\uv{\U{}}] The star of the image of $f$ in $\U{}$
    is equal to $Y$ and there exists an open cover $\W{}$ of $Y$ such
    that for each $W$ in $\W{}$ there exists $U$ in $\U{}$ such that
    the inclusion $f^{-1}(W) \subset f^{-1}(U)$ induces trivial (zero)
    homomorphisms on homotopy groups of dimensions less than~$n$,
    regardless of the choice of the base point.
  \end{enumerate}\em
\end{defin}

Our main result is the following theorem.
\begin{thm}\label{T:fne_nh}
  For each open cover $\U{}$ of an $n$-dimensional N\"obeling manifold
  $Y$ there exists an open cover $\V{}$ such that for each map $f$
  from an $n$-dimensional N\"obeling manifold into $Y$, if \fne{\V{}},
  then \nh{\U{}}.
\end{thm}

Theorem~\ref{T:fne_nh} is an analogue of theorems of Ferry on Hilbert
space and Hilbert cube manifolds~\cite{ferry1977} and of a theorem of
Chapman and Ferry on euclidean manifolds~\cite{chapmanferry1979}.

Let (P$_\U{}$), (Q$_\U{}$) and (R$_\U{}$) be any of the predicates
stated in definition~\ref{D:relative}.  We are interested in which of
the implications
\[
\forall_\U{} \exists_\V{} \forall_f (\text{P}_\V{}) \Rightarrow
(\text{Q}_\U{})
\]
are true. We show that if $Y$ is an $ANE(n)$-space, then, with
quantifiers understood to be same as above, \nh{\V{}} $\Rightarrow$
\ans{\U{}} (lemma~\ref{L:nh}), \ans{\V{}} $\Rightarrow$ \fne{\U{}}
(lemma~\ref{L:ans}) and \fne{\V{}} $\Rightarrow$ \uv{\U{}}
(lemma~\ref{L:fne}). These implications are standard.  To complete the
picture, we give an example that shows that \uv{\V{}} does not imply
\fne{\U{}}, even if $X$ and $Y$ are N\"obeling manifolds
(example~\ref{E:uv}).

Observe that we have the following rule of inference
\[
\left( \forall_\U{} \exists_\V{} \forall_f (\text{P}_\V{}) \Rightarrow
  (\text{Q}_\U{}) \right) \wedge \left( \forall_\U{} \exists_\V{}
  \forall_f (\text{Q}_\V{}) \Rightarrow (\text{R}_\U{}) \right)
\Rightarrow \left( \forall_\U{} \exists_\V{} \forall_f (\text{P}_\V{})
  \Rightarrow (\text{R}_\U{}) \right).
\]

Hence the above mentioned implications yield the following theorem.

\begin{thm}\label{T:relative nearh}
  For each open cover $\U{}$ of an $n$-dimensional N\"obeling manifold
  $Y$ there exists an open cover $\V{}$ such that for each map $f$
  from an $n$-dimensional N\"obeling manifold $X$ into $Y$ if
  \begin{center}
    \nh{\V{}} or \ans{\V{}} or \fne{\V{}},
  \end{center}
  then
  \begin{center}
    \nh{\U{}} and \ans{\U{}} and \fne{\U{}}.
  \end{center}
\end{thm}

Now consider absolute versions of conditions stated in
definition~\ref{D:relative}.

\begin{defin}
  For each map $f$ from a space $X$ into a space $Y$ we say that \nh{}
  (\ans{}, \fne{} or \uv{} respectively) is satisfied, if for each
  open cover $\U{}$ of $Y$ \nh{\U{}} (\ans{\U{}}, \fne{\U{}} or
  \uv{\U{}} respectively) is satisfied.

  If a map satisfies \nh{}, then we say that it is a
  \emph{near-homeomorphism}. If it satisfies \ans{}, then we say that
  it is \emph{approximately $n$-soft}. If it satisfies \fne{}, then we
  say that it is a \emph{fine $n$-homotopy equivalence}. If it
  satisfies \uv{}, then we say that it is a \emph{UV$^{n-1}$-map}.
\end{defin}

We shall show that if $Y$ is an $ANE(n)$-space, then \uv{}
$\Rightarrow$ \fne{} (lemma~\ref{L:uv}, which contrasts
example~\ref{E:uv}). Hence we have

\begin{center}
  \nh{} $\Rightarrow$ \ans{} $\Rightarrow$ \fne{} $\Leftrightarrow$
  \uv{}.
\end{center}

The above implications combined with theorem~\ref{T:fne_nh} yield the
following theorem.
\begin{thm}\label{T:nearh}
  The following conditions are equivalent for each map $f \colon X \to Y$ of
  $n$-dimensional N\"obeling manifolds:
  \begin{enumerate}
  \item[\nh{}] $f$ is a near-homeomorphism,
  \item[\ans{}] $f$ is approximately $n$-soft,
  \item[\fne{}] $f$ is a fine $n$-homotopy equivalence,
  \item[\uv{}] $f$ is an $UV^{n-1}$-map.
  \end{enumerate}
\end{thm}

\section{Preliminaries}
\label{S:preliminaries}

\begin{defin}
  We say that a metric space $X$ is an \emph{absolute neighborhood
    extensor in dimension $n$} if it is a metric space and if every
  map into $X$ from a closed subset $A$ of an $n$-dimensional metric
  space extends over an open neighborhood of $A$. The class of
  absolute neighborhood extensors in dimension $n$ is denoted by
  $ANE(n)$ and its elements are called \emph{$ANE(n)$-spaces}.
\end{defin}

\begin{lem}\label{L:nh}
  For each open cover $\U{}$ of an $ANE(n)$-space $Y$ there exists an
  open cover $\V{}$ such that if a map into $Y$ satisfies \nh{\V{}},
  then it satisfies \ans{\U{}}.
\end{lem}
\begin{proof}
  Choose open covers $\V{}$ and $\W{}$ of $Y$ such that the star of
  $\W{}$ refines $\U{}$ and the following condition is
  satisfied~\cite[Proposition 4.1.7]{chibook} for each at most
  $n$-dimensional metric space $B$ and its closed subset $A$:

  \vspace{ 1mm}
  \begin{tabular}{rl}
    (*) &
    \begin{tabular}{p{0.8\textwidth}}\noindent\em
      If one of two $\mathcal{V}$-close maps of $A$ into $Y$ has an
      extension to $B$, then the other also has an extension to $B$ and we
      may assume that these extensions are $\mathcal{W}$-close.
    \end{tabular}
  \end{tabular}
  \vspace{ 1mm}

  Let $A$ be a closed subset of an at most $n$-dimensional metric
  space~$B$ and let maps $\varphi \colon A \to X$ and $\psi \colon B
  \to Y$ be such that $f\varphi = \psi |A$. By~\nh{\V{}}, there exists
  a homeomorphism $g \colon X \to Y$, which is $\mathcal{V}$-close to
  $f$.  By the above stated property of $\mathcal{V}$ there exists a
  $\mathcal{W}$-close to $\psi$ extension $h \colon B \to Y$ of the
  composition $g\varphi$. Let $k = g^{-1}h \colon B \to X$.  Clearly,
  $k|A = \varphi$ and $fk$ is $\mathcal{U}$-close to $\psi$.
\end{proof}

\begin{defin}
  Let $\U{}$ be an open cover of a space $Y$. We say that maps $f, g
  \colon X \to Y$ are \emph{$n$-homotopic} if for every map~$\varPhi$
  from a polyhedron of dimension less than $n$ into $X$, the
  compositions $f \circ \varPhi$ and $g \circ \varPhi$ are
  homotopic by a homotopy whose paths refine~$\U{}$.
\end{defin}

\begin{lem}\label{L:ans}
  For each open cover $\U{}$ of an at most $n$-dimensional $ANE(n)$-space~$Y$ there exists an
  open cover $\V{}$ such that if a map into $Y$ satisfies \ans{\V{}},
  then it satisfies \fne{\U{}}.
\end{lem}
\begin{proof}
  Choose open covers $\V{}$ and $\W{}$ of $Y$ such that $\st_\V{} \st
  \W{}$ refines $\U{}$ and condition~(*) defined in the proof of
  lemma~\ref{L:nh} is satisfied.  By \ans{\V{}}, there exists a map $g
  \colon Y \to X$ such that $f \circ g$ is $\V{}$-close to the
  identity on $Y$. By (*), any two $\V{}$-close maps from an at most $n$-dimensional metric space 
  are $\st\W{}$-$n$-homotopic.  Hence $f \circ g$ is $\U{}$-$n$-homotopic to
  the identity on $Y$.  Let $k$ be a map into $X$ defined on an at
  most $(n-1)$-dimensional polyhedron $K$. Let $l = g \circ f \circ
  k$.  Since $f \circ g$ is $\V{}$-close to the identity on $Y$, $f
  \circ k$ is $\V{}$-close to $f \circ l$. By (*), there exists a
  $\st \W{}$-homotopy $H \colon K \times [0, 1] \to Y$ of $f \circ k$ and
  $f \circ l$. By \ans{\V{}}, this homotopy can be lifted to a
  homotopy of $k$ and $l$ in $Y$, whose composition with $f$ is
  $\V{}$-close to $H$. Since $\st_\V{} \st \W{}$ refines~$\U{}$, this
  composition is a $\U{}$-homotopy. Hence $H$ is a
  $f^{-1}(\U{})$-homotopy and $g \circ f$ is
  $f^{-1}(\U{})$-$n$-homotopic to the identity on $X$.
\end{proof}

\begin{lem}\label{L:fne}
  For each open cover $\U{}$ of an $ANE(n)$-space $Y$ there exists an
  open cover $\V{}$ such that if a map into $Y$ satisfies \fne{\V{}},
  then it satisfies \uv{\U{}}.
\end{lem}
\begin{proof}
  Let $\W{}$ be an open cover of $Y$ whose star refines $\U{}$. By
  theorem \cite[2.1.12]{chibook}, there exists an open cover
  $\V{}$ of $Y$ such that for each $V$ in $\V{}$ there
  exists $W_V$ in~$\W{}$, for which the inclusion $V \subset W_V$
  induces trivial homomorphisms on homotopy groups of dimensions less
  than~$n$. Let $k < n$. Let $V$ in $\V{}$. Let $\varphi \colon S^k
  \to f^{-1}(V)$. We will show that $\varphi$ is null-homotopic in
  $f^{-1}(\st_\V{} W_V)$, which shall end the proof, as $\st_\V{} \W{}$ refines~$\U{}$. By \fne{\V{}},
  there exists a map $g \colon Y \to X$ such that $g \circ f$ is
  $f^{-1}(\V{})$-$n$-homotopic with the identity on $X$ and $f \circ
  g$ is $\V{}$-close to the identity on $Y$.  In particular, $\varphi$
  is homotopic with $g \circ f \circ \varphi$ in
  $st_{f^{-1}(\V{})} f^{-1}(V) \subset f^{-1}(\st_\V{} V) \subset f^{-1}(\st_\V{} W_V)$. By the assumptions, $f \circ \varphi$ is null-homotopic in $W_V$. Hence $g \circ f \circ \varphi$ is null-homotopic in $g(W_V) \subset f^{-1}(\st_\V{} W_V)$. We are done.
\end{proof}

\begin{ex}\label{E:uv}
  We show that there exists a space $Y$ and an open cover $\U{}$ of
  $Y$ such that for each open cover $\V{}$ of $Y$ there exists a map
  $f$ from a space $X$ into $Y$ such that \uv{\V{}} is satisfied, but
  both \fne{\U{}} and \ans{\U{}} are not. We give an example for $n >
  1$ and the map that we construct is onto $Y$. For $n = 1$ an example
  can also be constructed, but the map cannot have a dense image
  in~$Y$.

  Let $Y$ be the unit interval $[0, 1]$ and let $\U{} = \{ [0, 1] \}$
  be the trivial cover of $Y$. Let~$\V{}$ be an open cover of $Y$ and
  let $V$ be an element of $\V{}$ that contains $\frac{1}{2}$. Let
  $\varepsilon > 0$ such that $[\frac{1}{2} - \varepsilon, \frac{1}{2}
  + \varepsilon] \subset V$. Let $X = [0,1] \times [0,1] \setminus [
  \frac{1}{2} - \varepsilon, \frac{1}{2} + \varepsilon] \times \{
  \frac{1}{2} \}$. Let $f \colon X \to Y$ be a restriction to $X$ of
  the projection of $[0,1] \times [0,1]$ onto the first coordinate.
  We can verify that $f$ satisfies \uv{\V{}} from the definition,
  taking any $\W{}$ that refines $\V{}$ and whose mesh is smaller than
  $2\varepsilon$.  As $X$ is homotopy equivalent to a circle and~$Y$
  is contractible, $f$ does not induce a monomorphism on fundamental
  groups of $X$ and~$Y$.  This is easily seen to contradict both
  \fne{\U{}} and \ans{\U{}} for $n > 1$.

  It is easy to modify the above example is such a way that $X$ and
  $Y$ are $n$-dimensional N\"obeling manifolds.
\end{ex}

\begin{lem}\label{L:uv}
  If a map into an at most $n$-dimensional $ANE(n)$-space satisfies \uv{}, then it satisfies
  \fne{}.
\end{lem}
\begin{proof}
  Let $f$ be a UV$^{n-1}$-map from a space $X$ into an $ANE(n)$-space
  $Y$. We will show by induction that for each $0 \leq k \leq n$, $f$
  satisfies \ans{} for polyhedral pairs $(A, B)$ such that $\dim A
  \setminus B \leq k$. Let $(P_k)$ denote the last condition.  For $k
  = 0$ the assertion is obvious, as \uv{} implies that $f$ has dense
  image in $Y$. Assume that $k > 0$.  Let $\U{}$ be an open cover of
  $Y$. By \uv{}, there exists an open cover $\W{}$ of $Y$ such that
  for each $W \in \W{}$ there exists $U_W \in \U{}$ such that the
  inclusion of $f^{-1}(W)$ into $f^{-1}(U_W)$ induces zero
  homomorphisms on homotopy groups of dimensions less than $n$. Let
  $\mathcal{S}$ be an open cover whose star refines $\W{}$. Let $B$ be
  a subpolyhedron of an at most $n$-dimensional polyhedron~$A$ such
  that $\dim A \setminus B \leq k$. Let maps $\varphi \colon B \to X$
  and $\psi \colon A \to Y$ be such that $f \circ \varphi =
  \psi_{|B}$.  Fix a triangulation of $A$ such that for each simplex
  $\delta$ of this triangulation $\psi(\delta) \subset S$ for some $S
  \in \mathcal{S}$.  By $(P_{k-1})$, we may extend $\varphi$ over the
  $(k-1)$-dimensional skeleton of $A$ to a map $k$ in such a way that
  $f \circ k$ is $\mathcal{S}$-close to $\psi$.  Consider an
  $k$-dimensional simplex $\delta$ of $A$.  Observe that $k$ maps
  boundary of $\delta$ into the inverse image $f^{-1}(W)$ of an
  element $W$ of $\W{}$.  Hence, $k$ extends over~$\delta$ to a map
  into the inverse image $f^{-1}(U_W)$.  Extend $k$ in this manner
  over all $k$-dimensional simplexes of $A \setminus B$ and observe
  that $f \circ k$ is $\U{}$-close to $\varphi$. This completes the
  inductive step and a proof that $(P_n)$ holds for $f$.

  Let $\U{}$ be an open cover of $Y$. We will show that \fne{\U{}} is
  satisfied. Choose open covers $\V{}$ and $\W{}$ of $Y$ such that
  $\st_\V{} \st \W{}$ refines $\U{}$ and condition (*) defined in the
  proof of lemma~\ref{L:nh} is satisfied. By
  \cite[Theorem 2.1.12(vii)]{chibook}, there exist an at most
  $n$-dimensional polyhedron $A$ and two maps $q \colon Y \to A$, $p\colon A \to Y$
  such that $p \circ q$ is $\V{}$-close to the identity on $Y$. By
  $(P_n)$, there exists a map $r \colon K \to X$ such that $f \circ r$
  is $\V{}$-close to $p$. Let $g = r \circ q$.  By the construction,
  $f \circ g$ is $\st \V{}$-close to the identity on $Y$. The rest of
  the proof follows the proof of lemma~\ref{L:ans}.
\end{proof}


\section{Proof that \fne{\V{}} implies \nh{\U{}}}
\label{S:hard implication}

For definitions of notions used throughout the proof we refer the
reader to~\cite{nagorkophd}. Let~$\U{}$ be an open cover of an
$n$-dimensional N\"obeling manifold~$Y$. Let~$f$ be a map from an
$n$-dimensional N\"obeling manifold~$X$ into~$Y$. Assume that~$q$ is
an integer greater that a constant~$m$ obtained by
lemma~\cite[8.4]{nagorkophd} applied to~$n$. Additionally assume that
$q$ is greater than $36(5N+8)^{n-1} + 3$, where~$N$ is a
constant obtained by theorem~\cite[6.7]{nagorkophd}.  By
lemma~\cite[8.1]{nagorkophd}, there exists a closed partition $\Q{} =
\{ Q_i \}_{i \in I}$ of~$Y$ that is $q$-barycentric and whose $q$th
star refines~$\U{}$. Observe that if $\P{} = \{ P_i \}_{i \in I}$ is a
closed partition of~$X$ that is isomorphic to~$\Q{}$, then by
lemma~\cite[8.4]{nagorkophd}, there exists a homeomorphism~$h$ of~$X$
and~$Y$ that maps elements of~$\P{}$ into the corresponding elements
of $\st^m \Q{}$. If $P_i \subset f^{-1}(\st^q_\Q{} Q_i)$ for each~$i$
in~$I$, then~$h$ is $\U{}$-close to~$f$. In this case $h$ is a homeomorphism that we are looking for. Let~$\V{}$ be an open cover
of~$Y$ whose star refines~$\Q{}$.  Assume that~$f$ satisfies
\fne{\V{}}. We will show that there exists a closed partition~$\P{}$
of~$X$ satisfying the above stated conditions. This shall end the
proof.

Our first goal is to construct an $n$-semiregular closed interior $\N{n}$-cover of $X$ that is isomorphic to $\Q{}$ and such that $f$ maps its elements into $l$th stars of the corresponding elements of $\Q{}$, for some constant $l$ (we obtain $l = 9$, but the exact value is of no importance). 
By proposition~\cite[6.4]{nagorkophd}, $\Q{}$ is $n$-semiregular.
Hence there exists an anticanonical map $\lambda$ of $\Q{}$ that is an
$n$-homotopy equivalence. By \fne{\V{}}, there exists a map $Y \to X$
whose composition with $f$ is $\V{}$-close to the identity. Hence
there exists a map $\hat \lambda$ into $X$ whose composition with $f$
is $\V{}$-close to $\lambda$.  Since $X$ is strongly universal in
dimension $n$, $\hat \lambda$ can be approximated by a closed
embedding $\Lambda$ whose composition with $f$ is $\st \V{}$-close to
$\lambda$. By the choice of $\V{}$, $\lambda$ is $\Q{}$-close to $f
\circ \Lambda$.  Hence the composition $\lambda \circ \Lambda^{-1}$ is
$\Q{}$-close to the restriction of $f$ to $\im \Lambda$. Hence
$\lambda(\Lambda^{-1}(f^{-1}(Q_i))) \subset \st_\Q{} Q_i$ for each $i$
in $I$. By lemma~\cite[6.16]{nagorkophd}, there exists an extension of
$\lambda \circ \Lambda^{-1}$ to a map $g$ from $X$ to $Y$ such that
$g(f^{-1}(Q_i)) \subset \st^7_\Q{} Q_i$ for each $i$ in $I$. This
implies that $g$ is $\st^7 \Q{}$-close to $f$ and that $g^{-1}(Q_i)
\subset f^{-1}(\st^8_\Q{} Q_i)$ for each $i$ in $I$.  Let $R_i =
g^{-1}(Q_i)$ for each $i$ in $I$. By the construction, $\R{} = \{ R_i
\}_{i \in I}$ is isomorphic to $\Q{}$. By
theorem~\cite[4.5]{nagorkophd}, there exists a closed interior
$\N{n}$-cover $\P{0} = \{ P^0_i \}_{i \in I}$ of $X$ that is a
swelling of $\R{}$. By taking a small enough swelling, we may require that $P^0_i \subset
f^{-1}(\st^9_\Q{} Q_i)$ for each $i \in I$. By the construction,
$\Lambda$ is an anticanonical map of $\P{0}$. By
lemma~\cite[2.20]{nagorkophd} and by corollary~\cite[5.1]{nagorkophd},
the composition $f \circ \Lambda$ is an $n$-homotopy equivalence,
since $f \circ \Lambda$ is $Q{}$-close to $\lambda$ and $\lambda$ is
an $n$-homotopy equivalence. Since $f$ is $n$-homotopy equivalence,
$\Lambda$ is $n$-homotopy equivalence and by the definition, $\P{0}$
is $n$-semiregular.

Hence we constructed a cover $\P{0}$ that satisfies the following condition.
\begin{enumerate}
\item[($0$)] $\P{0} = \{ P^0_i \}_{i \in I}$ is a closed star finite
  $0$-regular $n$-semiregular interior $\N{n}$-cover of $X$ that is
  isomorphic to $\Q{}$ and such that $P^0_i \subset f^{-1}(\st^9_\Q{}
  Q_i)$ for each $i$ in $I$.
\end{enumerate}

Our second goal is a construction of a sequence $\P{1}, \P{2}, \ldots, \P{n}$ of covers
of $X$ such that for each $0 < k \leq n$ the following condition is
satisfied.
\begin{enumerate}
\item[($k$)] $\P{k} = \{ P^k_i \}_{i \in I}$ is a closed star finite
  $k$-regular $n$-semiregular interior $\N{n}$-cover of $X$ that is
  isomorphic to $\Q{}$ and such that $P^k_i \subset
  f^{-1}(\st^{9(5N + 8)^k}_\Q{} Q_i)$ for
  each $i$ in $I$.
\end{enumerate}

Observe that $\P{} = \P{n}$ satisfies conditions stated at the
beginning of the proof. Hence a construction of such a sequence shall
finish the proof. 

Let $\F{}$ be a cover of $Y$ that refines $\st^l \Q{}$ for some positive integer $l < q$. Let $p$ be a positive integer such that $2^{p-2} - 1 < l \leq 2^{p-1} - 1$. By the assumption that $q$ is big enough, $\Q{}$ is $p$-barycentric. Hence by lemma~\cite[6.12]{nagorkophd} and lemma~\cite[6.15]{nagorkophd}, $\st^{2^{p-1} -1} \Q{}$ is $n$-contractible in $\st^{2^p - 1} \Q{}$. As $2^p - 1 < 4l+3$, $\F{}$ is $n$-contractible in $\st^{4l+3} \Q{}$. By \fne{\V{}}, if a cover $\F{}$ is $n$-contractible in a cover $\st^{4l+3} \Q{}$, then $f^{-1}(\F{})$ is $n$-contractible in $f^{-1}(\st_\V{} \st^{4l+3} \Q{}) \prec f^{-1}(\st^{4(l+1)} \Q{})$. Hence by theorem~\cite[6.7]{nagorkophd} applied to $\P{k}$, there exists a closed $k$-regular $n$-semiregular interior $\N{n}$-cover $\P{k}$ that is isomorphic to
$\P{k-1}$ and that refines $\st^N f^{-1}(\st^{4(9(5N + 8)^{k-1} + 1)} \Q{})$. By remark~\cite[3.5]{nagorkophd}, we may
require that $\P{k}$ is equal to $\P{k-1}$ on the image of $\Lambda$.
This implies that $P^{k}_i \subset \st^{N+1}
f^{-1}(\st^{4(9(5N + 8)^{k-1} + 1)} \Q{})$, hence $P^{k}_i \subset f^{-1}(\st^{N+1} \st^{4(9(5N + 8)^{k-1} + 1)} \Q{})$.  By lemma~\cite[2.1]{nagorkophd}, $\st^{N+1} \st^{4(9(5n+8)^{k-1} + 1)} \Q{} = \st^{(N+2)(4(9(5N+8)^{k-1}) + N + 1} \Q{}$ and clearly the latter exponent is not greater than $9(5N + 8)^k$. We are done.

\section{An alternative proof that \uv{} implies \nh{}}

It is known (see \cite[Proposition 5.7.4]{chibook}) that every
$n$-dimensional Menger manifold $M$ has the pseudo-interior
$\nu^{n}(M)$.

\begin{lem}\label{L:mpi}
  The class of $n$-dimensional N\"{o}beling manifolds coincides with
  the class of pseudo-interiors of $n$-dimensional Menger manifolds.
\end{lem}
\begin{proof}
  Apply \cite[Proposition 5.7.5]{chibook} and the open embedding
  theorem for N\"{o}beling manifolds \cite{levin}, \cite{nagorkophd}.
\end{proof}

Next we single out one of the main particular cases in which
near-ho\-meo\-mor\-phisms appear naturally.

\begin{pro}\label{P:inclusion}
  Let $A$ be a $\sigma Z$-set in a N\"{o}beling manifold $N$. Then the
  inclusion $N \setminus A \hookrightarrow N$ is a near-homeomorphism.
\end{pro}
\begin{proof}
  Apply Lemma \ref{L:mpi} and \cite[Proposition
  5.7.7]{chibook}.
\end{proof}

If for a map $f \colon X \to Y$ the image $f(X)$ is dense in $Y$ (as
is the case for approximately $n$-soft maps), then the set of
nondegenerate values of $f$, denoted by $N_{f}$, consists, by
definition, of three types of points of $Y$: points in $Y \setminus
f(X)$, points whose inverse images contain at least two points, and
points $y \in Y$ for which although the inverse image $f^{-1}(y)$ is a
singleton, the collection $f^{-1}({\mathcal B})$ does not form a local
base at $f^{-1}(y)$ for any local base ${\mathcal B}$ at $y$ in $Y$.
Note that $N_{f}$ is an $F_{\sigma}$-subset of $Y$ and that the
restriction $f|f^{-1}(Y\setminus N_{f}) \colon f^{-1}(Y\setminus
N_{f}) \to Y \setminus N_{f}$ is a homeomorphism.

Our next statement extends Proposition \ref{P:inclusion}.

\begin{pro}\label{P:sigmazset}
  Let $f \colon M\to N$ be an approximately $n$-soft map of
  $n$-dimensional N\"{o}beling manifolds. If $N_{f}$ is a $\sigma
  Z$-set, then $f$ is a near-homeomorphism.
\end{pro}
\begin{proof}
  We follow proof of \cite[Proposition 3.1]{bestvina}. Let $\{
  \alpha_{k} \colon k \in {\mathbb N}\} \subset C(I^{n},M)$ be a dense
  subset of embeddings of the $n$-dimensional cube into $M$ such that
  $\alpha_{i}(I^{n}) \cap \alpha_{j}(I^{n}) = \emptyset$ for each $i,j
  \in {\mathbb N}$ with $i \neq j$.

  The map $f_{0}\alpha_{1} \colon I^{n} \to N$ can be approximated by
  a $Z$-embedding $\beta_{1}\colon I^{n} \to N \setminus N_{f_{0}}$
  which is ${\mathcal U}-n$-homotopic to $f_{0}\alpha_{1}$ for a
  sufficiently small open cover ${\mathcal U}$ of $N$. Note that since
  $f_{0}\alpha_{1}$ and $f_{0}f_{0}^{-1}\beta_{1}$ are $n$-homotopic
  via a small "$n$-homotopy", the approximate $n$-softness of $f_{0}$
  implies that $\alpha_{1}$ and $f_{0}^{-1}\beta_{1}$ are
  $n$-homotopic in $M$ (via a small "$n$-homotopy", where smallness is
  measured in $N$). A version of $Z$-set unknotting theorem \cite[Theorem 2.2]{levin}
  produces a homeomorphism $h_{1}\colon M \to M$ such that
  $h_{1}\alpha_{1} = f_{0}^{-1}\beta_{1}$ and $f_{0}h_{1}$ is close to
  $f_{0}$. Let $f_{1} = f_{0}h_{1}$. Requiring additionally that
  $h_{1}$ is fixed outside of a small neighborhood of
  $f_{0}^{-1}(f_{0}(\alpha_{1}(I^{n})))$ we conclude that
  $f_{1}^{-1}(f_{1}(m)) = m$ for each $m \in \alpha_{1}(I^{n})$.
  Continuing in this manner we construct the sequence $f_{0} = f,
  f_{1}, \dots$ of approximately $n$-soft maps of $M$ into $N$ so that
  $f_{k+1} = f_{k}h_{k+1}$, where $h_{k+1} \colon M\to M$ is a
  homeomorphism fixed outside of a small neighborhood of
  $f_{k}^{-1}(f_{k}(\alpha_{k+1}(I^{n})))$ missing $\bigcup\{
  \alpha_{i}(I^{n}) \colon 1 \leq i \leq k\}$. As above,
  $h_{k+1}\alpha_{k+1} = f_{k}^{-1}\beta_{k+1}$, where $\beta_{k+1}
  \colon I^{n} \to N \setminus N_{f_{k}}$ is a $Z$-embedding. Observe
  also that $\bigcup\{ f_{k+1}(\alpha_{k}(I^{n}))\colon 1 \leq i \leq
  k+1\} \subseteq N \setminus N_{f_{k+1}}$, $f_{k}^{-1}(f_{k}(m)))=m$
  for each $m \in \bigcup\{ \alpha_{i}(I^{n})\colon 1 \leq i \leq k\}$
  and $f_{k}|\bigcup\{ \alpha_{i}(I^{n})\colon 1\leq i \leq k-1\} =
  f_{k-1}|\bigcup\{ \alpha_{i}(I^{n})\colon 1\leq i \leq k-1\}$. If
  the homeomorphism $h_{k+1}$ is chosen sufficiently close to $h_{k}$,
  then the map $g = \lim\{ f_{k}\} \colon M\to N$ will be
  approximately $n$-soft. Note that $g^{-1}(N_{g}) \cap \bigcup\{
  \alpha_{k}(I^{n}) \colon k \in {\mathbb N}\} = \emptyset$. It then
  follows from the choice of the set $\{ \alpha_{k}(I^{n}) \colon k
  \in {\mathbb N}\}$ that $g^{-1}(N_{g})$ is a $\sigma Z$-set in $M$.

  By Proposition \ref{P:inclusion}, both inclusions $i \colon M
  \setminus g^{-1}(N_{g})\hookrightarrow M$ and $j \colon N \setminus
  N_{g} \hookrightarrow N$ are near-homeomorphisms. Therefore, $g$
  (and hence $f$) can be approximated by a homeomorphism of the form
  $H_{j}g_{0}H_{i}^{-1}$, where $H_{i}$ approximates $i$, $H_{j}$
  approximates $j$ and $g_{0} = g|(M\setminus g^{-1}(N_{g})) \colon
  M\setminus g^{-1}(N_{g}) \to N\setminus N_{g}$.
\end{proof}

For a map $f \colon X \to Y$ and a closed subset $B \subseteq Y$ the
adjunction space $X \cup_{f}B$ is defined to be the disjoint union of
$X\setminus f^{-1}(B)$ and $B$ topologised as follows: $X\setminus
f^{-1}(B)$ itself is open and $f^{-1}(U\setminus B) \cup (U \cap B)$
for $U \subseteq Y$ open in $Y$. Obviously, the map $f$ factors as
follows:

\[
\xymatrix{
  X \ar^{f}[rr] \ar^{f_{B}}[dr] &  & Y \\
  & {X\cup_{f}B} \ar^{p_{B}}[ur] \\
}
\]
\noindent where $f_{B} \colon X \to X \cup_{f}B$ coincides with the
identity on $X\setminus f^{-1}(B)$ and with $f$ on $f^{-1}(B)$ and
$p_{B} \colon X \cup_{f}B \to Y$ coincides with the identity on $B$ and
with $f$ on $X\setminus f^{-1}(B)$. If $f$ is an approximate $n$-soft
map between Polish ${\rm ANE}(n)$-spaces and $B$ is a (strong) $Z$-set
in $Y$, then $X\cup_{f}B$ is also a Polish ${\rm ANE}(n)$-space
containing $B$ as a (strong) $Z$-set and both $f_{B}$ and $p_{B}$ are
approximately $n$-soft. If, in addition, $X$ and $Y$ are
$n$-dimensional N\"{o}beling manifolds, then so is $X \cup_{f}B$.

{\it Proof of \uv{}$\Rightarrow$\nh{}}. Proof follows the proof of
\cite[Characterization Theorem]{bestvina}.  Let $\{ \beta_{k}
\colon k \in {\mathbb N}\} \subset C(I^{n},M)$ be a dense subset of
embeddings of the $n$-dimensional cube into $N$ such that
$\beta_{i}(I^{n}) \cap \beta_{j}(I^{n}) = \emptyset$ for each $i,j \in
{\mathbb N}$ with $i \neq j$.

Let $f_{0}=f$ and consider the above described factorization of $f$
through the adjunction space, i.e. $f_{0} = p_{\beta_{1}(I^{n})}
f_{\beta_{1}(I^{n})}$

\[
\xymatrix{
  M \ar^{f_{0}=f}[rr] \ar_{f_{\beta_{1}(I^{n})}}[dr] &  & N \\
  & {M\cup_{f_{0}}\beta_{1}(I^{n})} \ar_{p_{\beta_{1}(I^{n})}}[ur] \\
}
\]

Note that $N_{f_{\beta_{1}(I^{n})}} \subseteq \beta_{1}(I^{n})$.
Since every compact subset of an $n$-dimensional N\"{o}beling manifold
is a strong $Z$-set (see \cite[Corollary 5.1.6]{chibook}), it follows
from Proposition \ref{P:sigmazset} that there exists a homeomorphism $h
\colon M \to M\cup_{f_{0}}\beta_{1}(I^{n})$ approximating
$f_{\beta_{1}(I^{n})}$ as close as we wish. Let $f_{1} =
p_{\beta_{1}(I^{n})}h$. Clearly, $f_{1}$ approximates $f_{0}$ and is
one to one over $\beta_{1}(I^{n})$. Proceeding in this manner we
construct a sequence $\{ f_{k} \colon k \in {\mathbb N}\}$ of
approximately $n$-soft maps (of $M$ into $N$) such that $f_{k+1}$
approximates $f_{k}$ and is one to one over $\bigcup\{
\beta_{i}(I^{n})\colon 1 \leq i \leq k+1\}$. If $f_{k+1}$ is
sufficiently close to $f_{k}$, then the limit map $g = \lim\{ f_{k}\}
\colon M\to N$ will be approximately $n$-soft. Since $g$ is one to one
over $\bigcup\{ \beta_{k}(I^{n}) \colon k \in {\mathbb N}\}$ it
follows from the choice of the collection $\{ \beta_{k}\}$ that
$N_{g}$ is a $\sigma Z$-set in $N$. By Proposition \ref{P:sigmazset},
$g$, and hence $f$, is a near-homeomorphism.


\begin{thebibliography}{99}



\bibitem{bestvina} M.~Bestvina, P.~Bowers, J.~Mogilsky, J.~Walsh, {\it Characterization of Hilbert space manifolds revisited}, Topology Appl. {\bf 24} (1986), 53--69.

\bibitem{chapmanferry1979} T.A.Chapman, S. Ferry, {\it Approximating homotopy equivalences by homeomorphisms}, Amer. J. Math. {\bf 101} (1986), 53--69.

\bibitem{chibook} A.~Chigogidze, {\it Inverse Spectra}, North Holland, Amsterdam, 1996.

\bibitem{ferry1977} S.~Ferry, {\it The homeomorphism group of a compact Hilbert manifold is an ANR}. Ann. Math., {\bf 106} (1977), 101--119.


\bibitem{levin} M.~Levin, {\it Characterizing N\"{o}beling spaces}, available onlite at
\url{http://front.math.ucdavis.edu/math.GT/0602361}

\bibitem{levin2}
M.~Levin, {\it A $Z$-set unknotting theorem for N\"obeling manifolds}, available online at \url{http://front.math.ucdavis.edu/math.GT/0510571}.



\bibitem{nagorkophd} A.~Nag\'orko,
{\it Characterization and topological rigidity of N\"{o}beling manifolds},
PhD Thesis, Warsaw University, 2006. Available online at 
\url{http://arxiv.org/abs/math/0602574}.

\bibitem{west}
J.~E.~West, {\it Open problems in infinite-dimensional topology}, 524--597; in: {\it Open Problems in Topology}, North Holland, Amsterdam, 1990 (edited by J. van Mill, G.M. Reed).

\end{thebibliography}
\end{document}